\documentclass[12pt,a4paper,oneside]{amsart} 

\usepackage[english]{babel}
\usepackage[utf8]{inputenc}
\usepackage{setspace}
\usepackage{enumitem}
\usepackage{microtype}
\usepackage[usenames,dvipsnames,svgnames,table]{xcolor}
\usepackage[textsize=footnotesize]{todonotes}

\usepackage{tikz}
\usepackage{tikz-cd}
\usetikzlibrary{arrows}
\usetikzlibrary{matrix}
\usetikzlibrary{calc}
\usepackage{tikz-qtree}
\usepackage{tikz-qtree-compat}

\usepackage{amsmath}
\usepackage{amsfonts}
\usepackage{amssymb}
\usepackage{amsthm}
\usepackage{amstext}
\usepackage{braket}
\usepackage{mathtools}
\usepackage{array}
\usepackage{enumitem}
\usepackage{mathrsfs}

\theoremstyle{plain}
\newtheorem{theorem}{\textbf{Theorem}}[section]
\newtheorem{proposition}[theorem]{\textbf{Proposition}}
\newtheorem{corollary}[theorem]{\textbf{Corollary}}	
\newtheorem{lemma}[theorem]{\textbf{Lemma}}

\theoremstyle{definition}

\newtheorem{remark}[theorem]{\textbf{Remark}}
\newtheorem{example}[theorem]{Example}

\theoremstyle{plain}

\theoremstyle{plain}

\newcommand{\thistheoremname}{}
\newtheorem*{genericthm*}{\thistheoremname}
\newenvironment{namedthm*}[1]
{\renewcommand{\thistheoremname}{#1}%
	\begin{genericthm*}}
	{\end{genericthm*}}
\newenvironment{namedtheorem*}[1]
{\renewcommand{\thistheoremname}{#1}%
	\begin{genericthm*}}
	{\end{genericthm*}}

\numberwithin{equation}{section}

\newcommand{\df}{\overset{\text{def}}{=}}

\newcommand{\lra}{\longrightarrow}

\newcommand{\pr}{\operatorname{pr}}

\usepackage[a4paper,top=3cm,bottom=3cm,left=0.5cm,right=0.5cm,bindingoffset=0mm]{geometry}
\usepackage{hyperref}
\usepackage[style=math-alphabetic,maxbibnames=99]{biblatex}
\bibliography{secant_arxiv_biblio.bib}

\title{The Martens-Mumford theorem and the Green-Lazarsfeld secant conjecture}
\author{Daniele Agostini}
\address{Max-Planck-Instit\"ut f\"ur Mathematik in den Naturwissenschaften, Inselstrasse 22, 04013 Leipzig, DE}
\email{daniele.agostini@mis.mpg.de}

\begin{document}

\begin{abstract}
   	The Green-Lazarsfeld secant conjecture predicts that the syzygies of a curve of sufficiently high degree are controlled by its special secants. We prove this conjecture for all curves of Clifford index at least two and not bielliptic and for all line bundles of a certain degree. Our proof is based on a classic result of Martens and Mumford on Brill-Noether varieties and on a simple vanishing criterion that comes from the interpretation of syzygies through symmetric products of curves. 
\end{abstract}

\maketitle

\section{Introduction}

The Green-Lazarsfeld secant conjecture predicts that the syzygies of a curve of sufficiently high degree are controlled by its special secants. More precisely, let $C$ be a smooth projective curve of genus $g$ and Clifford index $\operatorname{Cliff}(C)\geq c$, and let $L$ be a nonspecial line bundle of degree $\deg L = 2g + p + 1 -c$.
Then the conjecture states that 
\begin{equation}\label{conj:secant}
L \text{ has property } (N_p)  \qquad \text{ if and only if } \qquad L \text{ it is } (p+1)\text{ -very ample }.
\end{equation} 
We have stated here the only the nonspecial case of the conjecture, since the special case falls into Green's conjecture for canonical curves. Moreover, the ``only if'' direction is known and relatively straightforward to prove, see for example \cite[Theorem 4.36]{AproduNagelBook2010}, and the open, and more difficult part, is the converse.

The case $p=0$ and $c$ arbitrary was proven by Green and Lazarsfeld in \cite[Theorem  1]{GreenLazarsfeldProjective1986}. Instead, when $p$ is arbitrary the conjecture was proven by Green \cite[Theorem 4.a.1]{GreenKoszul1984}  for $c=0$ and by Green and Lazarsfeld \cite[Theorem 2]{GreenLazarsfeldFinite1988} for $c=1$. Furthermore, the conjecture was proven by Farkas and Kemeny \cite[Theorem 1.3]{FarkasKemenyGeneric2015} for every $p$ and $c$ under the assumption that both  the curve $C$ and the bundle $L$ are general. In the same paper, the authors give also more precise result that avoid the generality assumptions, such as \cite[Theorem 1.4, Theorem 1.5]{FarkasKemenyGeneric2015}. 

Here we would like to show how to attack the case $c=2$ of the Secant Conjecture, using classical results of Martens and Mumford that bound the dimensions of the Brill-Noether loci $W^r_d(C)$ in terms of the Clifford index. More precisely, Martens' Theorem asserts that if $C$ is a curve of Clifford index $\operatorname{Cliff}(C)\geq 1$ (meaning that $C$ is not hyperelliptic), then 
\begin{equation}\label{eq:martens}
\dim W^1_{g-1}(C)  = \rho(g,1,g-1) = g-4.
\end{equation}
Mumford refined this result as follows: suppose that the curve has Clifford index $\operatorname{Cliff}(C)\geq 2$ and is not bielliptic, then    
\begin{equation}\label{eq:mumford}
\dim W^1_{g-2}(C)  = \rho(g,1,g-2) = g-6.
\end{equation}
We have stated here just a partial version of the full results of Martens and Mumford, since it is the one that we will need. However, it turns out that the full results are equivalent to the versions given above \cite[Theorem 5.1, Theorem 5.2]{ACGH}.

Our result is the following:

\begin{namedthm*}{Main Theorem}
Let $C$ be a smooth curve of genus $g$ of Clifford index $\operatorname{Cliff}(C)\geq 2$ and not bielliptic. Then the Secant Conjecture holds for any line bundle of degree $\deg L = 2g+p-1$, i.e.
\begin{equation*}
L \text{ has property } (N_p)  \qquad \text{ if and only if } \qquad L \text{ it is } (p+1)\text{ -very ample }.
\end{equation*} 
\end{namedthm*}

Observe that this almost proves the Secant Conjecture for $c=2$: the only exception is that of bielliptic curves, which comes precisely from Mumford's Theorem \eqref{eq:mumford}.

Our strategy is based on the following simple vanishing result for syzygies: 

\begin{namedthm*}{Vanishing Criterion}
	Let $C$ be a smooth curve and $B,L$ two line bundles on $C$, with $H^0(C,B)=0$. Let also $p\geq 0$ and $D$ an effective divisor on $C$ of degree $m\leq p$.
	\begin{equation*}
	\text{If } \qquad K_{p-m,1}(C,B,L-D)=0 \qquad { then } \qquad K_{p,1}(C,B,L)=0
	\end{equation*} 
\end{namedthm*}

More generally, the same criterion holds without hypotheses on $B$ if we replace the Koszul cohomology with the Koszul cycles, see Remark \ref{remark:proofvanishingcriterion}. The proof of this result is fairly easy, and it can also be seen in the framework of Aprodu's projection method for syzygies, see Remark \ref{remark:projection}. However, what is more important is the idea of its proof: indeed, this can be reinterpreted as a very natural vanishing statement for tautological bundles on symmetric products, see Corollary \ref{cor:robcorollary}, and this leads to a more general version, see Lemma \ref{cor:basicvanishingsyzygies},  where we consider all divisors $D\in C_m$ at the same time, and which we actually need when proving the Main Theorem.     

Furthermore, the point of view of the symmetric product shows how to apply this criterion in order to attack the Secant Conjecture. Let us explain the strategy, starting with the case $c=0$. Then $C$ is a smooth curve of genus $g$ and $L$ a line bundle of degree $2g+p+1$, which is automatically nonspecial and $(p+1)$-very ample. The Secant Conjecture predicts that $K_{p,1}(C,L,L)=0$, and by duality \cite[Theorem 2.c.6]{GreenKoszul1984}, this is the same as $K_{g,1}(C,K_C-L,L)=0$. Since $L$ is nonspecial, we can apply the Vanishing Criterion with a general divisor $D\in C_g$, and then the Koszul complex shows that $K_{0,1}(C,K_C-L,L-D)=H^0(C,K_C-D)=0$.

Let us look at the next case, $c=1$. Now $C$ is not hyperelliptic and $L$ is of degree $2g+p$, nonspecial and $(p+1)$-very ample. Applying Koszul duality again we need to prove that $K_{g-1,1}(C,K_C-L,L)=0$. We can try to apply the Vanishing Criterion with a general divisor $D\in C_{g-2}$ and then we need to show $K_{1,1}(C,K_C-L,L-D)=0$. At this point, we can follow the proof in \cite{GreenLazarsfeldFinite1988}: first, since $C$ is not hyperelliptic, Martens' Theorem \eqref{eq:martens} implies that $K_C-D$ is globally generated, see Lemma \ref{lemma:c-1veryampleness}, so that we have an exact sequence
\begin{equation} 
0 \longrightarrow M_{K_C-D} \longrightarrow H^0(C,K_C-D)\otimes \mathcal{O}_C \longrightarrow \mathcal{O}_C(K_C-D) \longrightarrow 0 \end{equation}
then, we can identify, see Proposition \ref{prop:koszulcohomology}, $K_{1,1}(C,K_C-L,L-D) \cong H^0(C,M_{K_C-D}\otimes(L-D))$ and finally the base-point-free pencil trick shows that $M_{K_C-D} \cong D-L$, so that $H^0(C,M_{K_C-D}\otimes(L-D)) \cong H^0(C,L-K_C)=0$, where the last vanishing follows from the fact that $L$ is $(p+1)$-very ample, see Lemma \ref{lemma:p1veryampleness}.

Now we discuss the case of general $c\geq 1$: let $C$ be a smooth curve of genus $g$, of Clifford index $\operatorname{Cliff}(C)\geq c$ and let $L$ be a nonspecial line bundle on $C$, of degree $\deg L = 2g+p+1-c$ and $(p+1)$-very ample. By duality, the Secant Conjecture is equivalent to $K_{g-c,1}(C,K_C-L,L)=0$. Since $L$ is nonspecial, we can try to apply the vanishing criterion with a general divisor $D$ of degree $g-2c$: we need to show
\begin{equation}\label{eq:introDvanishing}
K_{c,1}(C,K_C-L,L-D)=0.
\end{equation}
Now we work on the symmetric product $C_{c}$, which parametrizes degree $c$ effective divisors on $C$. On the symmetric product we have the tautological bundle $E_{K_C-D}$ and the determinant bundle $N_{L-D} = \det E_{L-D}$. Voisin's interpretation of syzygies and an observation of Ein and Lazarsfeld \cite[Lemma 1.1]{EinLazarsfeldGonality2015} show that the syzygies are identified with the kernel of the multiplication map
\begin{small} 
\begin{equation}\label{eq:kermult}
K_{c,1}(C,K_C-L,L-D) \cong \operatorname{Ker} \left[ H^0(C_{c},E_{K_C-D})\otimes H^0(C_c,N_{L-D}) \to H^0(C_c,E_{K_C-D}\otimes N_{L-D}) \right].
\end{equation} 
\end{small}
Now, assume that 
\begin{equation}\label{eq:brillnoethercondition}
\dim W^1_{g-c}(C) = \rho(g,1,g-c) = g-2c-2.
\end{equation}
We have already seen that this is satisfied for $c=1$ thanks to Martens' Theorem \eqref{eq:martens}, and it is satisfied for $c=2$ if the curve is not bielliptic, thanks to Mumford's Theorem \eqref{eq:mumford}. In the general case, with this assumption one can prove that $E_{K_C-D}$ is globally generated, see Lemma \ref{lemma:c-1veryampleness}, so that there is an exact sequence
\begin{equation}\label{eq:exactsymmc}
0 \longrightarrow M_{E_{K_C-D}} \longrightarrow H^0(C,K_C-D)\otimes \mathcal{O}_{C_c} \longrightarrow E_{K_C-D} \longrightarrow 0
\end{equation}
and then \eqref{eq:kermult} shows that
\begin{equation}\label{eq:kermult2}
K_{c,1}(C,K_C-L,L-D) \cong H^0(C_c,M_{E_{K_C-D}}\otimes N_{L-D})
\end{equation}
Furthermore, the exact sequence \eqref{eq:kermult} yields a Buchsbaum-Rim resolution of $M_{E_{K_C-D}}\otimes N_{L-D}$, which is basically the generalization of the base-point-free pencil trick. Finally, assuming some cohomological vanishings on the symmetric product and the $(p+1)$-very ampleness of $L$, can then be used to prove the desired result $H^0(C_c,M_{E_{K_C-D}}\otimes N_{L-D})=0$.

The proof of our Main Theorem follows this strategy for $c=2$, and we can prove the required vanishing of \eqref{eq:kermult2} unless $g=p+4$, in which case we obtain $H^0(C_2,M_{E_{K_C-D}}\otimes N_{L-D}) \cong \wedge^4 H^0(C,K_C-D)\otimes H^0(C,D)$, which is unfortunately nonzero, but rather one-dimensional, for a general $D\in C_{g-4}$. However, letting $D$ vary, this means that a nonzero syzygy in $K_{p,1}(C,L,L)$ would give a section of a certain line bundle on $C_{g-4}$, but one can prove that, thanks to the assumption of $(p+1)$-very ampleness, this line bundle has no sections. Thus the proof is concluded in this case as well.

After proving the Vanishing Criterion and its generalizations, what is left to complete the strategy outlined above are the appropriate vanishing statements for cohomology of tautological bundles on symmetric products. In particular, these vanishings would prove the Secant Conjecture for every curve $C$ which satisfies the Brill-Noether conditions \eqref{eq:brillnoethercondition}. We can show that they hold for $c=2$, which gives our Main Theorem. Furthermore, the natural generalizations of these vanishing statements would allow to prove the Secant Conjecture for every curve which satisfies the Brill-Noether condition \eqref{eq:brillnoethercondition}: we include some comments about this at the end of the paper.

The paper is structured as follows: in Section \ref{sec:background} we recall some background material on $p$-very ampleness, symmetric products and syzygies. In Section \ref{section:vanishingresult} we prove the Vanishing Criterion of the Introduction, in a more general form. In Section \ref{section:secantconj}, we use it to prove our Main Theorem in all cases except $g=p+4$. Finally, in Section \ref{sec:global}, we globalize our strategy and we prove this last remaining case.

\vspace{10pt}

\textbf{Acknowledgments}: I would like to warmly thank Rob Lazarsfeld for generously sharing his ideas, that were fundamental for this project. I thank Gavril Farkas and Michael Kemeny for useful conversations. This work was started during a visit at Stony Brook University and I would like to thank their Department of Mathematics for the hospitality and the DAAD and the Berlin Mathematical School for the financial support.
\newpage 

\section{Background}\label{sec:background}

In this section we collect some facts that we will need later. 

\subsection{Higher order embeddings}

First recall that a line bundle $L$ on a smooth curve $C$ is $(p+1)$-very ample if and only if for every effective divisor $\xi\subseteq C$ of degree $p+2$, the evaluation map:
\begin{equation}\label{evLxi}
\operatorname{ev}_{L,\xi}\colon H^0(C,L) \lra H^0(C,L\otimes \mathcal{O}_{\xi})
\end{equation}
is surjective. Equivalently, $L$ fails to be $(p+1)$-very ample if the linear system $|L|$ embeds $C$ with a $(p+2)$-secant $p$-plane. Another characterization can be given via Riemann-Roch:

\begin{lemma}\label{kveryample}
	A line bundle $L$ on a smooth curve $C$ fails to be $(p+1)$-very ample if and only if there exists an effective divisor $\xi\subseteq C$ of degree $p+2$ such that $h^1(C,L-\xi) > h^1(C,L)$ or equivalently $h^0(C,K_C+\xi-L) > h^0(C,K_C-L)$.
\end{lemma}
\begin{proof}
	Immediate from Riemann-Roch.
\end{proof}

As a corollary, we can give the following reinterpretation of $(p+1)$-very ampleness for the bundles we are interested in:

\begin{lemma}\label{lemma:p1veryampleness}
	Let $C$ be a smooth curve of genus $g$ and $L$ a nonspecial line bundle of degree $\deg L = 2g+p+1-c$, with $c\geq 1$. Then $L$ is $(p+1)$-very ample if and only if the line bundle $2K_C-L$ is nonspecial and $(c-2)$-very ample.
\end{lemma}
\begin{proof}
	By Lemma \ref{kveryample}, $L$ fails to be $(p+1)$-very ample if and only if there exists an effective divisor $\xi\subseteq C$ of degree $p+2$ such that $h^0(C,K_C+\xi-L)>0$. Since $\deg(K_C+\xi-L) = c-1$, we can rephrase this by saying that there are two effective divisors $\xi,\xi'\subseteq C$ of degree $p+2$ and $c-1$ respectively such that $\xi' \in |K_C+\xi -L|$, or, equivalently $\xi \in |L+\xi'-K_C|$. Hence, $L$ is $(p+1)$-very ample if and only if for every effective divisor $\xi'\subseteq C$ of degree $c-1$ we have $h^0(C,L+\xi'-K_C)=0$, which is the same as $h^1(C,2K_C-L-\xi')=0$. Observe that since $\xi$ is effective, the vanishing $h^1(C,2K_C-L-\xi')=0$ implies that $h^1(C,2K_C-L)=0$ as well, and then it is easy to deduce the statement that we want from Lemma \ref{kveryample}.
\end{proof}

\begin{remark}\label{remark:pveryampleness}
	The previous lemma leaves out the case $c=0$, but any line bundle of degree $2g+p+1$ is nonspecial and $(p+1)$-very ample. Instead, in the case $c=1$ the previous statement should be interpreted by saying that $L$ is $(p+1)$-very ample if and only if $2K_C-L$ is nonspecial.
\end{remark}

\subsection{Symmetric products of curves and tautological bundles}
If $C$ is a smooth curve we will denote by $C_n$ its $n$-th symmetric product. This is a smooth and irreducible projective variety of dimension $n$ that parametrizes effective divisors of degree $n$ on $C$. As such, it comes equipped with the universal family:
\begin{equation}
\Xi_n\subseteq  C\times C_{n},  \qquad \Xi_n = \{ (x,\xi) \,|\,  x\in \xi \ \}%, \qquad p_n\colon \Xi_n \lra C_n, \qquad q_n\colon \Xi_n \lra C 
\end{equation}
%If we consider the maps
%\begin{align}
%j_{n} &\colon C\times C_{n-1} \lra C \times C_n & (x,\xi) &\mapsto (x,x+\xi) \\
%\sigma_n &\colon C\times C_{n-1} \lra C_n & (x,\xi) &\mapsto x+\xi 
%\end{align}
%we see that $j_n$ is a closed embedding such that $j_n(C\times C_{n-1}) \cong \Xi_n$, and we have a commutative diagram
%\begin{equation}
%p_n \circ j_n = \sigma_n % \red{FAI DIAGRAMMA}
%\end{equation} 
By construction, the fiber of the projection $\Xi_{n}\to C_n$ over a point $\xi\in C_n$ is isomorphic to the subscheme $\xi\subseteq C$.

For any line bundle $B$ on $C$ we can form the corresponding tautological bundle $E_B :\df \pr_{C_n,*}(\pr_C^*B\otimes \mathcal{O}_{\Xi_n})$ on $C_n$: this is a vector bundle of rank $n$ whose fiber at $\xi$ is identified with $H^0(C,B\otimes \mathcal{O}_{\xi})$. Tautological bundles come together with an evaluation map: indeed, pushing forward the exact sequence of sheaves on $C \times C_n$
\begin{equation}
0 \lra \pr_C^*B\otimes \mathcal{O}(-\Xi_n) \lra \pr_C^*B \lra \pr_C^*B\otimes \mathcal{O}_{\Xi_n} \lra 0
\end{equation}
we get an exact sequence on $C_n$:
\begin{equation}\label{evB}
0 \lra M_{E_{B}} \lra H^0(C,B)\otimes \mathcal{O}_{C_n} \overset{\operatorname{ev}_B}{\lra} E_B.
\end{equation}
By construction, the fiber of $\operatorname{ev}_B$ at $\xi \in C_n$ is exactly the evaluation map of \eqref{evLxi}. The map \eqref{evB} induces an isomorphism $H^0(C,B)\cong H^0(C_n,E_B)$ and in particular  it follows  that a line bundle $B$ on $C$ is $(n-1)$-very ample if and only if the tautological bundle $E_B$ on $C_{n}$ is globally generated. Moreover, the evaluation map  defines also a sheaf $M_{E_{B}}$ that we will denote as the kernel sheaf of $E_{B}$. From the definition, we see that $M_{E_B}$ is always a reflexive sheaf and it is moreover a bundle if $E_B$ is globally generated.

Of course, the determinant of the tautological bundle gives a line bundle $N_B=\det E_B$ on $C_n$. Other line bundles on $C_n$ can be constructed as follows: $C_n$ is obtained as the quotient $\pi\colon C^{n} \to C^{n}/\mathfrak{S}_n$, and for any line bundle $L$ on $C$ we can form the line bundle $L^{\boxtimes \, n} = \pr_1^* L \otimes \dots \otimes \pr_n^*L$ on $C^n$. Then we can define
\begin{equation}
S_L :\df \pi_*^{\mathfrak{S}_n}(L^{\boxtimes\, n}).
\end{equation}
One sees that $S_L$ is actually a line bundle, that $\pi^*S_L \cong L^{ \boxtimes n}$ and that the induced map
\begin{equation}
\operatorname{Pic}(C) \to \operatorname{Pic}(C_n) \qquad L \mapsto S_L
\end{equation}
is a homomorphism of groups. There are natural divisors associated with these line bundles: if we fix a point $x\in C$ then we have a natural divisor $S_x := x+C_{n-1} \subseteq C_n$ and the associated line bundle is precisely $S_{\mathcal{O}_C(x)}$. By linearity, we can define the divisor $S_D$ on $C_n$ for every divisor $D$ on $C_n$, and the associated line bundle is precisely $S_{\mathcal{O}_C(D)}$, so that we will use both notations interchangeably.

A distinguished divisor on $C_n$ is the locus $\Delta\subseteq C_n$ consisting of nonreduced divisors. The class of $\Delta$ is divisible by two in $\operatorname{Pic}(C_n)$: indeed, if we denote by $\delta$ the line bundle $\delta := N_{\mathcal{O}_C}^{\vee} = \det E_{\mathcal{O}_C}^{\vee}$, it turns out that $2\delta \cong \mathcal{O}_{C_n}(\Delta)$. More generally, it holds that
\begin{equation}\label{NL}
N_L \cong S_L -\delta.
\end{equation}
Moreover, the canonical bundle on $C$ is given by
\begin{equation}\label{canonical}
K_{C_n} \cong N_{K_C} = S_{K_C} - \delta.
\end{equation}

The cohomology of these line bundles is known:

\begin{lemma}\label{lemma:cohomologysymmetricproduct} 
	Let $L$ be an arbitrary line bundle on $C$.  Then we have isomorphisms
	\begin{equation}
		H^i(C_n,N_L) \cong \wedge^{n-i}H^0(C,L)\otimes \operatorname{Sym}^i H^1(C,L), \qquad H^i(C_n,S_L)\cong \operatorname{Sym}^{n-i} H^0(C,L) \otimes \wedge^{i}H^1(C,L) 
	\end{equation}	
\end{lemma}
\begin{proof}
	These follow from the discussion after \cite[Proposition 6.3]{KrugRemarks2018}.
\end{proof}

\subsection{Addition maps and incidence divisors}\label{section:additionmaps}

Let us fix two integers $m,n\geq 0$. Then we have a natural addition map:
\begin{equation}
\sigma = \sigma_{m,n}\colon C_m\times C_n \lra C_{m+n}, \qquad (D,E) \mapsto D+E.
\end{equation}
This is a finite and flat map, which is ramified precisely along the \emph{incidence divisor}
\begin{equation}
\Xi_{m,n} :\df \{ (D,E) \in C_m\times C_n \,|\, D\cap E \ne \emptyset \}.
\end{equation}
Observe that the incidence divisor $\Xi_{m,n}$ is irreducible, because it can be described as the image of the map 
\begin{equation}
C\times C_{m-1} \times C_{n-1} \lra C_m\times C_n, \qquad  (p,D',E') \mapsto (D'+p,E'+p)
\end{equation}
Moreover, when $m=1$, the incidence divisor $\Xi_{1,n}\subseteq C\times C_n$ coincides with the universal family $\Xi_n$ over $C_n$, and the map $C\times C_{n-1} \to \Xi_{1,n}$ is an isomorphism.

\subsection{Syzygies of curves and symmetric products} 

Let $L$ and $B$ be line bundles  on a smooth curve $C$. On the symmetric product $C_{p+1}$ we can twist the evaluation \eqref{evB} by $N_L$ to get a map
\begin{equation}\label{evBL}
\operatorname{ev}_{B,L}\colon H^0(C,B)\otimes N_L \to E_B\otimes N_L
\end{equation}
which can be used to compute syzygies. More precisely, define the set of Koszul cycles 
\begin{equation}
Z_{p,1}(C,B,L) :\df \operatorname{Ker} \left[ d_{p,1}\colon \wedge^p H^0(C,L) \otimes H^0(C,B+L) \to \wedge^{p-1}H^0(C,L) \otimes H^0(C,B+2L) \right]
\end{equation}
as the kernel of the Koszul differential. Then a result of Voisin, together with an observation of Ein and Lazarsfeld gives the following:

\begin{proposition}[Voisin, Ein-Lazarsfeld]\label{prop:koszulcohomology}
	The Koszul differential $d_{p,1}$ is identified with the multiplication map
	\begin{equation}
	H^0(C_p,N_L) \otimes H^0(C_p,E_{B+L}) \lra H^0(C_p,E_{B+L}\otimes N_L).
	\end{equation}
	Moreover, we also have:
	\begin{align}
		Z_{p,1}(C,B,L) &\cong H^0(C_{p+1},E_B\otimes N_L), \\
		K_{p,1}(C,B,L) &\cong \operatorname{Coker} \left[ H^0(C,B)\otimes H^0(C_{p+1},N_L) \to H^0(C_{p+1},E_B\otimes N_L) \right].
	\end{align}
\end{proposition}
\begin{proof}
	See \cite[Lemma 1.1]{EinLazarsfeldGonality2015}.
\end{proof}

\section{A vanishing result for syzygies}\label{section:vanishingresult}

%Here we want to describe Lazarsfeld's idea for the nonspecial secant conjecture. Thus, fix $c\geq 0$ and let $C$ be a curve of genus $g$, of Clifford index $\operatorname{Cliff}(C)\geq c$ and let $L$ be a nonspecial line bundle on $C$ of degree $d=2g+p+1-c$. By Riemann-Roch, we
%have
%\begin{equation}
%h^0(C,L) = d+1-g = g+p+2-c.
%\end{equation}

%Suppose that $L$ is $(p+1)$-very ample. The conjecture asserts that $K_{p,1}(C,L,L)=0$. By duality \cite[Theorem 2.c.6]{GreenKoszul1984}, this is the same as $K_{g-c,1}(C,K_C-L,L)=0$ and by Proposition \ref{prop:koszulcohomology}, we can rewrite this as
%\begin{equation}\label{noglobalsections}
%H^0(C_{g-c+1},E_{K_C-L}\otimes N_L ) = 0.
%\end{equation}
%So, it is a matter of checking whether this vector bundle has global sections or not. Lazarsfeld's idea is to use a simple vanishing result  that we are going to describe now. We start with some preliminaries about tautological bundles.

In this section, we prove the Vanishing Criterion from the introduction, in a more general form.

\subsection{Tautological bundles and addition maps}
Let us fix a smooth curve $C$ and two integers $m,n\geq 0$. We will need to know the behavior of tautological bundles under the  addition map $\sigma\colon C_m\times C_n \to C_{m+n}$. If $E_B$ is a tautological bundle on $C_{n+m}$, the fiber of the bundle $\sigma^*E_B$ over $(D,E)$ is given by construction by $H^0(C,B\otimes \mathcal{O}_{D+E})$. Now we observe that for every two effective divisors $D,E$ on $C$ we have a natural exact sequence of sheaves on $C$:
\begin{equation}
0 \lra B\otimes \mathcal{O}_{D\cup E} \lra (B\otimes \mathcal{O}_{D}) \oplus (B\otimes \mathcal{O}_E) \lra B\otimes \mathcal{O}_{D\cap E} \lra 0.
\end{equation}
Together with the canonical surjection $B\otimes \mathcal{O}_{D+E}\lra B\otimes \mathcal{O}_{D\cup E}$, this induces another exact complex 
\begin{equation}\label{mayervietoris}
 B\otimes \mathcal{O}_{D+E} \lra (B\otimes \mathcal{O}_{D}) \oplus (B\otimes \mathcal{O}_E) \lra B\otimes \mathcal{O}_{D\cap E} \lra 0
\end{equation}
As $D,E$ vary in $C_m\times C_n$ we can glue these exact sequences together to obtain the following:
\begin{lemma}\label{lemma:globalmayervietoris}
There is a short exact sequence of sheaves on $C_m\times C_n$:
\begin{equation}\label{globalmayervietoris}
0 \lra \sigma^*E_B \lra \pr_{C_m}^*E_B \oplus \pr_{C_n}^*E_B \lra \mathcal{J}_B^{m,n} \lra 0
\end{equation} 
which globalizes the sequence \eqref{mayervietoris}, meaning that the fiber of $\mathcal{J}_B^{m,n}$ on $(D,E)$ is identified with $B\otimes \mathcal{O}_{D\cap E}$. Moreover $\det \mathcal{J}_B^{m,n} \cong \mathcal{O}_{C_m\times C_n}(\Xi_{m,n})$, so that $\sigma^* N_{B} \cong \pr_{C_m}^*N_B \otimes \pr_{C_n}^*N_{B} \otimes \mathcal{O}(-\Xi_{m,n})$.
\end{lemma}
\begin{proof}
	Consider the two universal varieties $\Xi_{1,m}\subseteq C\times C_m$ and $\Xi_{1,n} \subseteq C\times C_n$ and their pullbacks to $C\times C_m\times C_n$, that we keep denoting with the same symbols. On $C\times C_m\times C_n$ there is an exact sequence
	\begin{equation}
	0 \lra \pr_C^*B\otimes \mathcal{O}_{\Xi_{m,1}\cup \Xi_{n,1}} \lra (\pr_C^*B\otimes \mathcal{O}_{\Xi_{m,1}}) \oplus (\pr_C^*B\otimes \mathcal{O}_{\Xi_{n,1}}) \lra B\otimes \mathcal{O}_{\Xi_{m,1}\cap \Xi_{n,1}} \lra 0
	\end{equation}
	Pushing forward to $C_m\times C_n$,this yields the exact sequence
	\begin{equation}
	0 \lra \pr_{C_m\times C_n,*}(\pr_C^*B\otimes \mathcal{O}_{\Xi_{m,1}\cup \Xi_{n,1}}) \lra \pr_{C_m}^*E_B \oplus \pr_{C_n}^*E_B \lra \mathcal{J}_B^{m,n} \lra 0
	\end{equation} 
	where $\mathcal{J}_B^{m,n} := \pr_{C_m\times C_n,*}(B\otimes \mathcal{O}_{\Xi_{m,1}\cap \Xi_{n,1}})$. Observe that this second sequence is exact on the right because the restriction of $\pr_{C_m\times C_n}$ to $\Xi_{1,m}\cup \Xi_{1,n}$ is a finite map, so that $R^1\pr_{C_m\times C_n,*}(\pr_C^*B\otimes \mathcal{O}_{\Xi_{m,1}\cup \Xi_{n,1}})=0$. 
	Now we need to prove that $\sigma^*E_B \cong  \pr_{C_m\times C_n,*}(\pr_C^*B\otimes \mathcal{O}_{\Xi_{m,1}\cup \Xi_{n,1}})$: consider the cartesian diagram
	\begin{equation}
	\begin{tikzcd}
	C\times C_m \times C_n \arrow{r}{\operatorname{id}\times \sigma}\arrow{d}{\pr_{C_m\times C_n}} & C\times C_{m+n} \arrow{d}{\pr_{C_{m+n}}} \\
	C_m\times C_n \arrow{r}{\sigma}  & C_{m+n} 
	\end{tikzcd}
	\end{equation}
	and observe that by flat base change we get isomorphisms, $\sigma^*E_B \cong \sigma^* \pr_{C_m\times C_n,*} (\pr_C^*B\otimes \mathcal{O}_{\Xi_{1,m+n}}) \cong \pr_{C_m\times C_n,*}((\operatorname{id}\times \sigma)^*(\pr_C^*B\otimes \mathcal{O}_{\Xi_{1,m+n}})) \cong \pr_{C_m\times C_n,*}(\pr_C^*B\otimes(\operatorname{id}\times \sigma)^*( \mathcal{O}_{\Xi_{1,m+n}}))$. At this point, it is easy to see that $(\operatorname{id}\times \sigma)^*\mathcal{O}_{\Xi_{1,m+n}} \cong \mathcal{O}_{\Xi_{1,m}\cup \Xi_{1,n}}$, which gives us what we want. As a consequence, we see that \eqref{globalmayervietoris} is truly the globalization of \eqref{mayervietoris}: indeed, taking fibers in \eqref{globalmayervietoris}, we get the sequence
	\begin{equation}
	 B\otimes \mathcal{O}_{D+E} \lra B\otimes \mathcal{O}_D \oplus B\otimes \mathcal{O}_E \lra {\mathcal{J}^{m,n}_B} \otimes \kappa(D,E) \lra 0
	\end{equation}
	and \eqref{mayervietoris} shows at that  ${\mathcal{J}^{m,n}_B} \otimes \kappa(D,E) \cong B\otimes \mathcal{O}_{D\cap E}$.
	The next step in the proof is to show that $\det \mathcal{J}^{m,n}_B \cong \mathcal{O}(\Xi_{m,n})$: however, it is straightforward to check that $\mathcal{J}^{m,n}_B$ is a coherent sheaf of rank one on the irreducible and reduced divisor $\Xi_{m,n}\subseteq C_m\times C_n$, so that we can apply Lemma \ref{lemma:determinant} below.  The final statement about $\sigma^*N_B$ follows because $\sigma^*N_B \cong \sigma^* \det E_B \cong \det \sigma^*E_B$, and we can compute this via the exact sequence \eqref{globalmayervietoris}.
\end{proof}

\begin{lemma}\label{lemma:determinant}
	Let $X$ be a smooth variety, $D\subseteq X$ a reduced divisor and $\mathcal{F}$ a coherent sheaf on $D$ of rank $r$ at each irreducible component of $D$. The determinant of $\mathcal{F}$ as a sheaf on $X$ is given by $\det \mathcal{F} \cong \mathcal{O}_X(rD)$. 
\end{lemma}
\begin{proof}
 This is a standard result, but we include a proof here for the sake of reference. Since $X$ is smooth and both $\det\mathcal{F}$ and $\mathcal{O}_X(rD)$ are line bundles, it is enough to show that they are isomorphic outside any closed set of codimension at least two. Thus, if we consider the set of points in $D$ where the rank of $\mathcal{F}$ jumps, we can reduce to the case where $\mathcal{F}$ is locally free of rank $r$ on $D$. Furthermore, if we consider an open subset $U\subseteq D$ such that $\mathcal{F}_{|U}$ is free, the complement $Z=D\setminus U$ has codimension at least two in $X$, so we can assume that $\mathcal{F}$ is free, that is $\mathcal{F} \cong \mathcal{O}_D^{\oplus r}$. At this point, we have the exact sequence of sheaves on $X$:
 \begin{equation}
 0 \lra \mathcal{O}_X(-D)^{\oplus r} \lra \mathcal{O}_X^{\oplus r} \lra \mathcal{O}_D^{\oplus r} \lra 0
 \end{equation} 
 which proves that $\det \mathcal{O}_D^{\oplus r} \cong \det \mathcal{O}_X^{\oplus r}\otimes (\det\mathcal{O}_X(-D)^{\oplus r})^{\vee} \cong \mathcal{O}_X(rD)$.
\end{proof}

Lemma \ref{lemma:globalmayervietoris} readily implies the general vanishing result that we will apply later to syzygies. %the main idea here is due to Rob Lazarsfeld.

\begin{lemma}\label{lemma:basicvanishing}
	Let $B$ be a line bundle on $C$ and $\mathscr{L}$ be an arbitrary line bundle on $C_{m+n}$. Consider the addition map $\sigma\colon C_m\times C_n \lra C_{m+n}$ and suppose that $H^0(C_m\times C_n,\pr_{C_n}^*E_B\otimes \sigma^*\mathscr{L}) = 0$. Then $H^0(C_{m+n},E_B\otimes \mathscr{L})=0$ as well.
\end{lemma}
\begin{proof}
	Let $s\in H^0(C_{n+m},E_B\otimes \mathscr{L})$ be a global section. We want to show that $s$ vanishes at a general point $\xi\in C_{n+m}$. We can write this point as $\xi = x_1+\dots+x_{m+n}$ where the $x_i$ are pairwise distinct: then we have a canonical decomposition of the fiber
	\begin{equation}
	(E_B\otimes \mathscr{L})\otimes \kappa(\xi) \cong \bigoplus_{i=1}^{m+n} H^0(C,B\otimes \mathcal{O}_{x_i}) \otimes \mathscr{L}(\xi)
	\end{equation}
	and the evaluation of $s$ at $\xi$ decomposes accordingly as $s(\xi) = s_1+\dots+s_{n+m}$. Thus, we want to prove that $s_i=0$ for every $i$: for this, let $\xi = D+E$ be an arbitrary decomposition into two divisors $D\in C_m$ and $E\in C_n$, then we have a corresponding decomposition $s(\xi) = s_D+s_E$, and  it is enough to prove that $s_E=0$. In other words, we have an evaluation map
	\begin{equation}\label{eq:evmapsE}
	H^0(C_{n+m},E_B\otimes \mathscr{L}) \lra H^0(C,B\otimes \mathcal{O}_E)\otimes \mathscr{L}(\xi), \qquad s \mapsto s_E
	\end{equation} 
	and we want to show that this is identically zero. To do so,  we can look at this  through the addition map. Indeed, consider the pullback $\sigma^*(s)$ as a section of $\sigma^*(E_B\otimes \mathscr{L})$ on $C_n\times C_m$: the sequence of Lemma \ref{lemma:globalmayervietoris} gives an embedding
	\begin{equation}
	\sigma^*(E_B\otimes \mathscr{L}) \hookrightarrow (\pr_{C_n}^*E_B\otimes \sigma^*\mathscr{L}) \oplus (\pr_{C_m}^*E_B\otimes \sigma^*\mathscr{L}) 
	\end{equation}
	which corresponds precisely to the decomposition $s(\xi)=s_D+s_E$ when applied to the evaluation of $\sigma^*(s)$ at $(D,E)$. Hence, the evaluation map \eqref{eq:evmapsE}
	factors through
	\begin{equation}
	H^0(C_{n+m},E_B\otimes \mathscr{L}) \overset{\sigma^*}{\hookrightarrow} H^0(C_n\times C_m,\sigma^*(E_B\otimes \mathscr{L})) \lra H^0(C_n\times C_m,\pr_{C_n}^*E_B\otimes \sigma^*\mathscr{L}) 
	\end{equation}
	and since $H^0(C_n\times C_m,\pr_{C_n}^*E_B\otimes \sigma^*\mathscr{L})=0$ by assumption, it follows that \eqref{eq:evmapsE} is identically zero.
\end{proof}

For computing syzygies, we look at the global sections of $E_B\otimes N_L$: in this situation, Lemma \ref{lemma:basicvanishing} translates into the following.

\begin{corollary}[Global Vanishing Criterion]\label{cor:basicvanishingsyzygies}  
Let $B$ and $L$ be line bundles on a curve $C$ and assume that 
\begin{equation}
H^0(C_m\times C_n, \pr_{C_m}^*N_L \otimes \operatorname{pr}_{C_n}^*(E_B\otimes N_L)\otimes  \mathcal{O}(-\Xi_{n,m})) = 0.
\end{equation} 
Then $H^0(C_{n+m},E_B\otimes N_L)=0$ as well. In particular, we have the vanishing of Koszul cohomology $K_{n+m-1,1}(C,B,L)=0$.
\end{corollary}
\begin{proof}
	This follows immediately from Lemma \ref{lemma:basicvanishing} and the computation of $\sigma^*N_L$ in Lemma \ref{lemma:globalmayervietoris}.
\end{proof}

A natural way to compute the global sections of $\pr_{C_m}^*N_L \otimes \operatorname{pr}_{C_n}^*(E_B\otimes N_L)\otimes  \mathcal{O}(-\Xi_{n,m})$ is to pushforward to $C_m$: by the projection formula, this yields the sheaves
\begin{align}
\mathscr{F} & :=\pr_{C_m,*}(\pr_{C_m}^*N_L \otimes \operatorname{pr}_{C_n}^*(E_B\otimes N_L)\otimes  \mathcal{O}(-\Xi_{m,n})) \cong N_L \otimes \mathscr{G}, \label{eq:sheafF} \\
\mathscr{G} & := \pr_{C_m,*}( \operatorname{pr}_{C_n}^*(E_B\otimes N_L)\otimes  \mathcal{O}(-\Xi_{m,n})). \label{eq:sheafG}
\end{align}
One checks easily that the restriction of $\operatorname{pr}_{C_n}^*(E_B\otimes N_L)\otimes  \mathcal{O}(-\Xi_{m,n})$ to a fiber $\{D\}\times C_n$ is given by 
\begin{equation}
E_B\otimes (N_L-S_D) \cong E_B\otimes N_{L-D}.
\end{equation} 
Hence, $\mathscr{G}$ bundles together the spaces $H^0(C_n,E_B\otimes N_{L-D})$. This remark yields the 
the following result, which is essentially the Vanishing Criterion from the Introduction.

\begin{corollary}[Vanishing Criterion]\label{cor:robcorollary}
	Let $B,L$ be line bundles on $C$ and assume that for an effective divisor $D\in C_{m}$ we have
	\begin{equation}%\label{vanishingonpullback}
	H^0(C_{n},E_B\otimes N_{L-D}) = 0.
	\end{equation}
	Then $H^0(C_{n+m},E_B\otimes N_L)=0$ as well. In particular, we have the vanishing of Koszul cohomology $K_{n+m-1,1}(C,B,L)=0$. 
\end{corollary}
\begin{proof}
We want to apply Corollary \ref{cor:basicvanishingsyzygies}, and, with the notation of the previous discussion, we need to show that $H^0(C_m,\mathscr{F}) = H^0(C_m,\mathscr{G}\otimes N_L) = 0$. Observe that both $\mathscr{F}$ and $\mathscr{G}$ are torsion-free sheaves, since they are pushforwards of locally free sheaves. Now we look at the point $D\in C_n$: by assumption, we have $H^0(C_n,E_B\otimes N_{L-D})=0$, and then the previous discussion, together with base change, show that the fiber of $\mathscr{G}$ at $D$ vanishes: $\mathscr{G}\otimes \kappa(D) = 0$. Thus, $\mathscr{G}$ is supported away from $D$, but since it is torsion-free, it must be that $\mathscr{G}=0$, and a fortiori $H^0(C_m,\mathscr{G}\otimes N_L) = 0$. 
\end{proof}

\begin{remark}\label{remark:proofvanishingcriterion}
	Let's just see explicitly that this corollary corresponds to the Vanishing Criterion: let $p\geq 0$ and $D\in C_m$ an effective divisor with $m\leq p$. Then Lemma \ref{prop:koszulcohomology} shows that $Z_{p,1}(C,B,L) \cong H^0(C_{p+1},E_B\otimes N_L)$ and $Z_{p-m,1}(C,B,L-D) \cong H^0(C_{p+1-m},E_B\otimes N_{L-D})$. Hence we can reinterpret Corollary \ref{cor:robcorollary} as a statement on Koszul cycles, by saying that if $Z_{p-m,1}(C,B,L-D)=0$ then $Z_{p,1}(C,B,L)=0$ as well. Observe that this is true without any requirements on $B$ and $L$. 
	
	If $H^0(C,B)=0$, then $Z_{p-m,1}(C,B,L-D)=K_{p-m,1}(C,B,L-D)$ and $Z_{p,1}(C,B,L)=K_{p,1}(C,B,L)$ so we get exactly the Vanishing Criterion.
\end{remark}

\begin{remark}\label{remark:generalizationgreen}
	Corollary \ref{cor:robcorollary}, o generalizes Green's vanishing theorem \cite[Theorem 3.a.1]{GreenKoszul1984} for line bundles on curves. Indeed, Green's vanishing states that if we have $h^0(C,B+L)\leq p$, then $Z_{p,1}(C,B,L)=0$. However, if $h^0(C,B+L)\leq p$, then we can find an effective divisor of degree $D$ such that $h^0(C,B+L-D)=0$, and then we see that $Z_{0,1}(C,B,L-D) = H^0(C,B+L-D)$, so the vanishing $Z_{p,1}(C,B,L)=0$ follows from Corollary \ref{cor:robcorollary}.
\end{remark}

\begin{remark}\label{remark:projection}
We can also interpret Corollary \ref{cor:robcorollary} in terms of projection maps of syzygies. Let's rephrase it in the terms of Remark \ref{remark:proofvanishingcriterion}: if $Z_{p-m,1}(C,B,L-D)=0$ then $Z_{p,1}(C,B,L)=0$ as well. It is clear that this statement can be reduced to the case where $D=x$ consist of a single point: assume for simplicity that the point $x$ is not a base point of $L$. Then the kernel bundles $M_L,M_{L-x}$ as in \eqref{evB} fit into an exact sequence
\begin{equation}
0 \lra M_{L-x} \lra M_L \lra \mathcal{O}_C(-x) \lra 0
\end{equation}
which in turn induces another exact sequence
\begin{equation}\label{eq:kernelbundleswedges}
0 \lra \wedge^p M_{L-x} \lra \wedge^p M_L  \lra \wedge^{p-1}M_{L-x}\otimes \mathcal{O}_C(-x) \lra 0
\end{equation} 
Standard results about kernel bundles \cite[Proof of Proposition 2.4]{AproduNagelBook2010} show that we have canonical identifications $Z_{p,1}(C,B,L) \cong H^0(C,\wedge^p M_L \otimes (B+L))$. Hence, tensoring the exact sequence \eqref{eq:kernelbundleswedges} with $B+L$ and taking global sections we get an exact sequence
\begin{equation}
0 \lra Z_{p,1}(C,B+x,L-x) \lra Z_{p,1}(C,B,L) \overset{\pr_x}{\lra} Z_{p-1,1}(C,B,L-x)
\end{equation}
The map $\operatorname{pr}_x$ is the so-called \emph{projection map} for syzygies, and it has been much studied, especially by Aprodu \cite{AproduVanishing2002},\cite{AproduNagelBook2010} and more recently by Kemeny \cite{KemenyProj20}. Thus, our statement seems to suggest the following: suppose that $Z_{p,1}(C,B,L)\ne 0$, then there exists an $\alpha\in Z_{p,1}(C,B,L)$ such that $\pr_x(\alpha)\ne 0$. In particular, this could most certainly be proved with Aprodu's tecnhiques as in \cite{AproduVanishing2002}, however we feel that the proof of Corollary \ref{cor:robcorollary} via tautological bundles gives a different, and useful,  point of view on the geometry of the problem. Moreover, it is not immediately clear to us how to get to the more general Corollary \ref{cor:basicvanishingsyzygies} using projection maps.
\end{remark}

\section{The secant conjecture}\label{section:secantconj}

In this section, we turn to the Secant Conjecture. We have already presented our strategy in the Introduction, but we recall it here, filling in the details. Thus, fix $c\geq 0$ and let $C$ be a curve of genus $g$ and of Clifford index $\operatorname{Cliff}(C)\geq c$, and let $L$ be a nonspecial line bundle on $C$ of degree $d=2g+p+1-c$. By Riemann-Roch, we have $h^0(C,L) = g+p+2-c$.

Suppose that $L$ is $(p+1)$-very ample. The conjecture asserts that $K_{p,1}(C,L,L)=0$. By duality \cite[Theorem 2.c.6]{GreenKoszul1984}, this is the same as $K_{g-c,1}(C,K_C-L,L)=0$ and since $K_C-L$ is special, we can attempt to use the Vanishing Criterion with a general divisor $D\in C_{g-2c}$. So, we need to prove that $K_{c,1}(C,K_C-L,L-D)=0$. 

\begin{example}[The case $c=0$]
	When $c=0$,  we need to show
	\begin{equation}
	K_{0,1}(C,K_C-L,L-D) = H^0(C,K_C-D) = 0
	\end{equation}
	for a general $D\in C_g$, but this is obviously true.
\end{example}

%by Proposition \ref{prop:koszulcohomology}, we can rewrite this as
%\begin{equation}\label{noglobalsections}
%H^0(C_{g-c+1},E_{K_C-L}\otimes N_L ) = 0
%\end{equation}
%To prove this, Lazarsfeld proposes to specialize Corollary \ref{cor:basicvanishingsyzygies} and Corollary \ref{cor:robcorollary} to this situation, by considering the addition map
%\begin{equation}
%\sigma\colon C_{c+1}\times C_{g-2c} \lra C_{g-c+1}
%\end{equation}
%Then, Corollary \ref{cor:basicvanishingsyzygies} asserts that if,
%\begin{equation}\label{eq:basicvanishingsecant}
%H^0(C_{c+1}\times C_{g-2c},\pr_{C_{c+1}}^*(E_{K_C-L}\otimes N_L) \otimes \pr_{C_{g-2c}}^*N_L \otimes \mathcal{O}(-\Xi_{c+1,g-2c})) = 0
%\end{equation}
%Then the vanishing \eqref{noglobalsections} holds. Moreover, Corollary \ref{cor:robcorollary}, shows that to get \eqref{eq:basicvanishingsecant} it suffices to prove
%\begin{equation}\label{eq:robvanishingsecant}
%H^0(C_{c+1},E_{K_C-L}\otimes N_{L-D}) = 0
%\end{equation} 
%for a single divisor $D\in C_{g-2c}$. In particular, this reproves already the Secant Conjecture for $c=0$: 

In the rest, we consider the cases $c\geq 1$. Then, we can use Proposition \ref{prop:koszulcohomology} and look at $K_{c,1}(C,K_C-L,L-D)$ as the kernel of the multiplication map:
\begin{equation}
H^0(C_c,E_{K_C-D})\otimes H^0(C_c,N_{L-D}) \lra H^0(C_c,E_{K_C-D}\otimes N_{L-D})
\end{equation}
Under a suitable Brill-Noether condition on $C$, we can express this in terms of a kernel bundle. 

\begin{lemma}\label{lemma:c-1veryampleness}
	Fix an integer $c\geq 1$  and  let $C$ be  a smooth curve of genus $g$ such that 
	\begin{equation}
	\dim W^1_{g-c}(C)=\rho(g,1,g-c) = g-2c-2.
	\end{equation} 
	Then for a general $D\in C_{g-2c}$ the line bundle $K_C-D$ is $(c-1)$-very ample, hence we get the exact sequence on $C_c$:
		\begin{equation}
		0 \longrightarrow M_{E_{K_C-D}} \longrightarrow H^0(C,K_C-D)\otimes \mathcal{O}_{C_c} \longrightarrow E_{K_C-D} \longrightarrow 0 
		\end{equation}
	and moreover
	\begin{equation}\label{eq:isotoM}
	K_{c,1}(C,K_C-L,L-D) \cong H^0(C_c,M_{E_{K_C-D}}\otimes N_{L-D})
	\end{equation}
\end{lemma}
\begin{proof}
	By Lemma \ref{kveryample}, $K_C-D$ fails to be $(c-1)$-very ample  precisely when there exists a $\xi\in C_{c}$ such that $h^0(D+\xi) > h^0(C,D)=1$. Thus, let 
	\begin{equation}
	Z = \{ D\in C_{g-2c} \,|\, h^0(D+\xi) \geq 2 \,\,\text{ for a certain } \xi\in C_c \}
	\end{equation}
	we need to prove that $Z$ is a proper subset of $C_{g-2c}$. By definition, $Z$ is the image of the locus
	\begin{equation}
	\Sigma = \{ (H,\xi) \in C_{g-c}\times C_c  \,|\, h^0(C,H)\geq 2, \, \xi \leq H\}
	\end{equation}
	under the difference map $\Sigma \to Z, (H,\xi) \mapsto H-\xi$; in particular $\dim Z \leq \dim \Sigma$. To estimate $\dim \Sigma$, we observe that the projection $\Sigma \to C_{g-c}$ is clearly finite onto its image, which is $C^1_{g-c} := \{ H\in C_{g-c} \,|\, h^0(C,H) \geq 2 \}$. Hence, $\dim \Sigma = \dim C^1_{g-c}$. Finally, $C^1_{g-c}$ mapso onto the Brill-Noether locus $W^1_{g-c}(C)$ under the Abel-Jacobi map $u\colon C^1_{g-c} \to W^1_{g-c}(C)$, and by hypothesis $\dim W^1_{g-c}(C)\leq g-2c-2$. Furthermore, we know from \cite[Lemma III.3.5]{ACGH} that the general fibers of the Abel-Jacobi map over every irreducible component of $W^1_{g-c}(C)$ have dimension one. This shows that $\dim C^1_{g-c} = \dim W^1_{g-c}(C)+1 \leq g-2c-1$: hence $\dim Z \leq g-2c-1$, so that it is a proper subset of $C_{g-2c}$.
\end{proof} 

We can try to compute the group $H^0(C_c,M_{E_{K_C-D}}\otimes N_{L-D})$ via the Buchsbaum-Rim complex of $M_{E_{K_C-D}}$:

\begin{lemma}\label{lemma:resolution}
	With the same hypotheses of Lemma \ref{lemma:c-1veryampleness} we have an exact complex on $C_c$:
	\begin{small}
	\begin{equation}
0 \to \begin{matrix} \wedge^{2c}H^0(C,K_C-D) \\ \otimes \\  \operatorname{Sym}^{c-1}E_{K_C-D}^{\vee}\otimes S_{L-K_C} \end{matrix} \to  \dots \to \begin{matrix}\wedge^{c+2}H^0(C,K_C-D) \\ \otimes \\ E_{K_C-D}^{\vee}\otimes S_{L-K_C} \end{matrix} \to \begin{matrix}\wedge^{c+1}H^0(C,K_C-D) \\ \otimes \\  S_{L-K_C} \end{matrix} \to M_{E_{K_C-D}} \otimes N_{L-D} \to 0 
	\end{equation}
	\end{small}
\end{lemma}
\begin{proof}
	This follows from the Buchsbaum-Rim complex of \cite[Theorem B.2.2]{LazarsfeldPositivityI2004}, applied to the short exact sequence of Lemma \ref{lemma:c-1veryampleness}.  
\end{proof} 

We apply now this strategy to the cases $c=1$ and $c=2$ of the Secant Conjecture.

\subsection{The secant conjecture for $c=1$}
When $c=1$, the curve $C$ is not hyperelliptic and the line bundle $L$ has degree $2g+p$, so that it is automatically nonspecial. Furthermore, since $C$ is not hyperelliptic, Martens' Theorem \ref{eq:martens} gives $\dim W^1_{g-1}(C)=\rho(g,1,g-1)$, so that Lemma \ref{lemma:c-1veryampleness} applies, and  we need to show that
\begin{equation} 
K_{1,1}(C,K_C-L,L-D) \cong H^0(C,M_{K_C-D}\otimes (L-D)) = 0
\end{equation}
for a general divisor $D\in C_{g-2}$. Then, Lemma \ref{lemma:resolution} gives an isomorphism
\begin{equation}
M_{K_C-D}\otimes (L-D) \cong \wedge^2 H^0(C,K_C-D)\otimes (L-K_C) \cong L-K_C.
\end{equation}
Observe that this isomorphism is exactly the base-point-free pencil trick that was used in the original proof of \cite[Theorem 3.3]{GreenLazarsfeldFinite1988}, and one can regard the Buchsbaum-Rim complex as its natural generalization. To conclude, we use the assumption that $L$ is $(p+1)$-very ample: indeed, by Remark \ref{remark:pveryampleness} and duality, the $(p+1)$-very ampleness of $L$ is precisely equivalent to $H^0(C,L-K_C)=0$.

\subsection{The secant conjecture for $c=2$}\label{section:c2}

Let us turn to the new case $c=2$. Then $C$ is a smooth curve of genus $g$ and Clifford index $\operatorname{Cliff}(C)\geq 2$, and $L$ is a $(p+1)$-very ample line bundle of degree $d = 2g+p-1$; in particular $L$ is automatically nonspecial. Now, \emph{assume that the curve $C$ is not bielliptic}: then Mumford's theorem \ref{eq:mumford} shows that $\dim W^1_{g-4}(C) = \rho(g,1,g-4)$, so that we can apply Lemma \ref{lemma:c-1veryampleness}, and then we need to show that
\begin{equation}
K_{2,1}(C,K_C-L,L-D) \cong H^0(C_2,M_{E_{K_C-D}}\otimes N_{L-D}) = 0
\end{equation}
where $D\in C_{g-4}$ is a general divisor. To do this, we can apply Lemma \ref{lemma:resolution}, and get a resolution:
\begin{equation}\label{eq:resolutionMc2} 
0 \lra \begin{matrix}\wedge^4 H^0(C,K_C-D) \\ \otimes \\ E^{\vee}_{K_C-D} \otimes S_{L-K_C} \end{matrix} \lra \begin{matrix}\wedge^3 H^0(C,K_C-D)\\ \otimes \\ S_{L-K_C}\end{matrix} \lra M_{E_{K_C-D}} \otimes N_{L-D} \lra 0
\end{equation}
Now we can use this resolution to compute global sections. We need a preliminary Lemma: 

\begin{lemma}\label{lemma:lemmah1}
	Let $L$ be a globally generated line bundle and $B$ another line bundle such that $H^1(C,L)=H^1(C,B+L)=0$. Then
	\begin{equation}
	H^1(C_{p+1},E_B\otimes N_L) \cong K_{p-1,2}(C,B,L).
	\end{equation} 	
\end{lemma}
\begin{proof}
	Consider the universal family $\Xi_{p+1}$ as the image of the closed embedding $\sigma\colon C\times C_p \to C_{p+1}$: by the projection formula, $E_B\otimes N_L \cong \sigma_{*}(\pr_C^*B\otimes \sigma^* N_L)$, and Lemma \ref{lemma:globalmayervietoris} shows that $\sigma^*N_L \cong \pr_C^*L \otimes \pr_{C_{p}}^*L \otimes \mathcal{O}_{C\times C_p}(-\Xi_{1,p})$. Since the map $\sigma$ is finite, it follows that
	\begin{equation}
	H^1(C_{p+1}, E_B\otimes N_L) \cong H^1(C\times C_p,\pr_C^*(B+L)\otimes \pr_{C_p}^*N_L \otimes \mathcal{O}_{C\times C_p}(-\Xi_{1,p})).
	\end{equation}
	The sheaf appearing in the right hand side is the kernel of the surjective map
	\begin{equation}
	\pr_C^*(B+L)\otimes \pr_{C_p}^*N_L \lra (\pr_C^*(B+L)\otimes \pr_{C_p}^*N_L)_{|\Xi_{1,p}} \lra 0 
	\end{equation}
	furthermore, K\"unneth's formula, together with our assumptions and Lemma \ref{lemma:cohomologysymmetricproduct}, yields that $
	H^1(C\times C_p,\pr_C^*(B+L)\otimes \pr_{C_p}^*N_L) = 0$. 
	%$
	%H^1(C\times C_p,\pr_C^*(B+L)\otimes \pr_{C_p}^*N_L) = \left[H^1(C,B+L)\otimes H^0(C_p,N_L)\right] \oplus \left[H^0(C,B+L)\otimes H^1(C_p,N_L)\right] = 0 
	%$ 
	Hence, taking cohomology we get 
	\begin{align*}
		H^1(C_p,&E_B\otimes N_L) \cong H^1(C\times C_p,\pr_C^*(B+L)\otimes \pr_{C_p}^*N_L \otimes \mathcal{O}_{C\times C_p}(-\Xi_{1,p})) \\
		& \cong \operatorname{Coker} \left[ H^0(C\times C_p,\pr_C^*(B+L)\otimes \pr_{C_p}^*N_L) \lra H^0(\Xi_{1,p},\pr_C^*(B+L)\otimes \pr_{C_p}^*N_L)_{|\Xi_{1,p}}) \right] \\
		& \cong K_{p-1,2}(C,B,L).
	\end{align*}
	where the last isomorphism comes from Voisin's interpretation of syzygies through the Hilbert scheme: see for example \cite[Corollary 5.5]{AproduNagelBook2010}.
\end{proof}

Finally we compute the global sections $H^0(M_{E_{K_C-D}} \otimes N_{L-D})$: 

\begin{lemma}\label{lemma:c2specialD}
With these assumptions, we have
\begin{equation}
H^0(C_2,M_{E_{K_C-D}} \otimes N_{L-D}) \cong \begin{cases} 0 & \text{ if } g\ne p+4 \\ \wedge^{4}H^0(C,K_C-D)\otimes H^0(C,D) & \text{ if } g = p+4 \end{cases}
\end{equation}
\end{lemma}
\begin{proof}
First we observe that since $L$ is $(p+1)$-very ample,  Lemma \ref{lemma:p1veryampleness} and duality give $H^0(C,L-K_C)=0$, and then Lemma \ref{lemma:cohomologysymmetricproduct} proves the vanishings
\begin{align}
H^0(C_2,S_{L-K_C}) & = \operatorname{Sym}^2 H^0(C, L-K_C) = 0, \\
H^1(C_2,S_{L-K_C}) & = H^0(C,L-K_C)\otimes H^1(C,L-K_C) = 0. 
\end{align}
Hence, taking cohomology in \eqref{eq:resolutionMc2} we get the isomorphism
\begin{equation}
H^0(C_2,M_{E_{K_C-D}}\otimes N_{L-D}) \cong \wedge^4 H^0(C,K_C-D)\otimes H^1(C_2,E_{K_C-D}^{\vee}\otimes S_{L-K_C}).
\end{equation}
Now we need to compute this group: Serre's duality gives
\begin{equation}
H^1(C_2,{E^{\vee}_{K_C-D}}\otimes S_{L-K_C}) \cong H^1(C_2,E_{K_C-D}\otimes S_{K_C-L}\otimes N_{K_C})^{\vee} = H^1(C_2,E_{K_C-D}\otimes N_{2K_C-L})^{\vee}
\end{equation} 
and Lemma \ref{lemma:lemmah1} shows that $H^1(C_2,E_{K_C-D}\otimes N_{2K_C-L}) \cong K_{0,2}(C,K_C-D,2K_C-L)$. Since $L$ is $(p+1)$-very ample, we know that $2K_C-L$ is nonspecial and globally generated from Lemma \ref{lemma:p1veryampleness}, so that  Green's duality theorem \cite[Theorem 2.c.6]{GreenKoszul1984} gives $K_{0,2}(C,K_C-D,2K_C-L) \cong K_{g-p-4,0}(C,D,2K_C-L)^{\vee}$. In summary, we found an isomorphism
\begin{equation}
H^0(C_2,M_{E_{K_C-D}} \otimes N_{L-D}) \cong \wedge^4 H^0(C,K_C-D)\otimes K_{g-p-4,0}(C,D,2K_C-L).
\end{equation}
To conclude, we look at the group $K_{g-p-4,0}(C,D,2K_C-L)$: since $H^0(C,L-2K_C-D)=0$ by degree reasons, it follows that this syzygy group is the kernel of the Koszul differential 
\begin{equation}
d:\wedge^{g-p-4}H^0(C,2K_C-L)\otimes H^0(D) \lra \wedge^{g-p-5}H^0(C,2K_C-L)\otimes H^0(D)
\end{equation}
Since $D$ is general, we have $h^0(C,D)=1$ and then the Koszul differential $d$ is injective, as soon as $g\ne p+4$. If instead $g=p+4$ the kernel of the Koszul differential is $H^0(C,D)$, so that $H^0(C_2,M_{E_{K_C-D}} \otimes N_{L-D}) \cong \wedge^4 H^0(C,K_C-D)\otimes H^0(C,D)$.
\end{proof}

This last Lemma proves the Main Theorem in all cases but $g=p+4$. However, it also gives us a hint on how we might proceed in this case: we give an informal description here which we will develop in a proper proof in the next section. The key idea is to use the General Vanishing Criterion of Lemma \ref{cor:basicvanishingsyzygies}: that result shows that a nonzero syzygy in $K_{p,1}(C,L,L)$ produces a nonzero element in $H^0(C_{g-4},\mathscr{G}\otimes N_L)$, where $\mathscr{G}$ is the pushforward of $\operatorname{pr}_{C_{3}}^*(E_B\otimes N_L)\otimes  \mathcal{O}(-\Xi_{g-4,3})$ from $C_{g-4}\times C_3$. Then, the previous discussion shows that the fiber of $\mathscr{G}$ at a general point $D$ is identified with 
$H^0(C_2,M_{E_{K_C-D}} \otimes N_{L-D}) \cong \wedge^{4}H^0(C,K_C-D)\otimes H^0(C,D)$. Observe that $\wedge^{4}H^0(C,K_C-D)\otimes H^0(C,D)$ is one-dimensional so we expect that $\mathscr{G}$ is a line bundle, at least over an appropriate open subset of $C_{g-4}$. A bit of reflection shows that this line bundle must then be $\mathscr{G}\cong \det M_{E_{K_C}} \cong N_{K_C}^{\vee}$, so that, in conclusion, we would get a nonzero element in
\begin{equation}
H^0(C_{g-4},\mathscr{G}\otimes N_L) \cong H^0(C_{g-4},N_L-N_{K_C}) \cong H^0(C_{g-4},S_{L-K_C})
\end{equation}
but this is impossible, because $H^0(C_{g-4},S_{L-K_C}) \cong \operatorname{Sym}^{g-4}H^0(C,L-K_C)$ and $H^0(C,L-K_C)=0$ from Lemma \ref{lemma:p1veryampleness}.

In the next section, we make this reasoning precise.

\section{Global computations}\label{sec:global}

Our task now is to take the proofs in Sections \ref{section:secantconj} and work them out as the divisor $D$ varies. We can do it in the general setting of the Secant Conjecture: let $c\geq 1$ be an integer, $C$ a smooth curve of genus $g$ and Clifford index $\operatorname{Cliff}(C)\geq c$ and $L$ a nonspecial line bundle of degree $\deg L = 2g+p+1-c$ which is $(p+1)$-very ample. We want to use the General Vanishing Criterion of Lemma \ref{lemma:basicvanishing}, and prove that 
\begin{equation}
H^0(C_{c+1}\times C_{g-2c}, \pr_{C_{c+1}}^*(E_{K_C-L}\otimes N_L)\otimes \pr_{C_{g-2c}}^*N_L \otimes \mathcal{O}(-\Xi_{c+1,g-2c}) ) = 0
\end{equation}
By pushing forward to $C_{g-2c}$, we get the sheaves
\begin{equation}
\mathscr{F} := \pr_{C_{g-2c},*}(\pr_{C_{c+1}}^*(E_{K_C-L}\otimes N_L)\otimes  \mathcal{O}(-\Xi_{c+1,g-2c}))\otimes N_L =: \mathscr{G}\otimes N_L
\end{equation}
and what we need to show is that $H^0(C_{g-2c},\mathscr{F}) \cong H^0(C_{g-2c}, \mathscr{G}\otimes N_L ) = 0$. Since these sheaves are torsion-free, it suffices to prove the vanishing on the open subset
\begin{equation}
U_{g-2c} := \{ D\in C_{g-2c} \,|\, h^0(C,D) = 1 \}
\end{equation}
\begin{remark}\label{remark:Ularge}
We observe that this open subset is large: indeed, its complement $C^1_{g-2c}$ has dimension $\dim C^1_{g-2c} \leq \dim W^1_{g-2c}(C)+1$ by \cite[Lemma III.3.5]{ACGH}. Since the curve $C$ is not hyperelliptic, Martens' theorem, in the form of \cite[Theorem 5.1]{ACGH} shows that $\dim W^1_{g-2c}(C)\leq g-2c-3$, so that $C^1_{g-2c}$ has codimension at least $2$ in $C_{g-2c}$. 

In particular, since $U_{g-2c}$ is large, we will not denote  explicitly the restrictions of vector bundles to $U_{g-2c}$ or $C_{c+1}\times U_{g-2c}$, because a vector bundle on a smooth variety is uniquely determined by its restriction to any large open subset. 
\end{remark}

\textbf{Step 1:} We expect the fiber of $\mathscr{G}$ at $D\in U$ to be $H^0(C_{c+1},E_{K_C-L}\otimes N_{L-D})$, and Proposition \ref{prop:koszulcohomology} shows that $H^0(C_{c+1},E_{K_C-L}\otimes N_{L-D}) \cong H^0(C_c,M_{E_{K_C-D}}\otimes N_{L-D})$. The first step is to globalize this isomorphism in terms of $\mathscr{G}$. To do so, we will globalize the exact sequence
\begin{equation}
0 \to M_{E_{K_C-D}} \lra H^0(C,K_C-D) \otimes \mathcal{O}_{C_{c}} \lra E_{K_C-D}
\end{equation}
with respect to $D$. More precisely, consider on $C\times C_c \times U_{g-2c}$ the exact sequence
\begin{small}
	\begin{equation}\label{eq:globallemma2}
	0 \lra \pr_C^*K_C\otimes \mathcal{O}(-\Xi_{1,g-2c}-\Xi_{1,c}) \lra \pr_C^*K_C\otimes \mathcal{O}(-\Xi_{1,g-2c}) \lra \left( \pr_C^*K_C\otimes \mathcal{O}(-\Xi_{1,g-2c}) \right)_{|\Xi_{1,c}} \lra 0
	\end{equation}
\end{small}
and let
\begin{equation}\label{eq:sheafM}
 0 \lra \mathscr{M} \lra \mathscr{H} \lra \mathscr{E}
\end{equation}
be its pushforward along $\pr_{C_c\times U_{g-2c}}$. Observe that if we restrict the sequence \eqref{eq:globallemma2} to the fibers $C\times \{(\xi,D)\}$, and we take global sections, we get the sequence 
\begin{equation}
H^0(C,K_C-\xi-D) \lra H^0(C,K_C-D) \lra H^0(C,(K_C-D)\otimes \mathcal{O}_{\xi}).
\end{equation}
By our choice of $U$, we see that $h^0(C,K_C-D)=2c$ for all $D\in U$, hence Grauert's theorem gives that  $\mathscr{H}$ and $\mathscr{E}$ are vector bundles on $C_{c}\times U$ of ranks $2c$ and $c$ whose fibers at $(\xi,D)$ are identified with $H^0(C,K_C-D)$ and $H^0(C,(K_C-D)\otimes \mathcal{O}_D)$ respectively. This allows us also to identify the restrictions of $\mathscr{H}$ and $\mathscr{E}$ to the subvarieties $C_c\times \{D\}$: indeed we have the cartesian diagram  
\begin{equation}
\begin{tikzcd}
C\times C_c \times \{D\} \arrow[hook]{r} \arrow{d}{\pr_{C_c}} & C\times C_c \times U \arrow{d}{\pr_{C_c\times U}} \\
C_c\times \{D\} \arrow[hook]{r}  & C_c\times U 
\end{tikzcd}
\end{equation}
and by cohomology and base change we see that 
\begin{align}
	\mathscr{H}_{|C_c\times \{D\}} & \cong \pr_{C_c,*}(\left(\pr_C^*K_C\otimes \mathcal{O}(-\Xi_{1,g-2c})\right)_{|C\times C_c \times \{D\}})\\
	& \cong \pr_{C_c,*}(\pr_C^*(K_C-D)) \cong H^0(C,K_C-D)\otimes \mathcal{O}_{C_c}
\end{align} 
and
\begin{align} 
	\mathscr{E}_{|C_c\times \{D\}} & \cong \pr_{C_c,*}\left(\left( \pr_C^*K_C\otimes \mathcal{O}(-\Xi_{1,g-2c}) \right)_{|\Xi_{1,c}}\right)_{|C\times C_c \times \{D\}} \\
	&\cong \pr_{C_c,*}(\pr_C^*(K_C-D)\otimes \mathcal{O}_{\Xi_{1,c}}) \cong E_{K_C-D}.
\end{align}
Hence, the sequence \eqref{eq:sheafM} on $C_c\times U$  globalizes the evaluation map $H^0(C,K_C-D)\otimes \mathcal{O}_{C_c} \to E_{K_C-D}$
on $C_c$, as $D$ varies in $U$.

Now it is easy to express $\mathscr{G}$ in terms of the sheaf $\mathscr{M}$ from \eqref{eq:sheafM}.

\begin{lemma}\label{lemma:GwithM}
	With the previous notations, we have
	\begin{equation*}
	\mathscr{G}_{|U} \cong \pr_{U,*}\left( \mathscr{M} \otimes \pr_{C_c}^*N_L\otimes \mathcal{O}(-\Xi_{c,g-2c})  \right).
	\end{equation*}
\end{lemma}
\begin{proof}
 We consider the sum map $\sigma\colon C\times C_c \to C_{c+1}$ and the cartesian diagram
\begin{equation*}
\begin{tikzcd}
C\times C_c \times U \arrow{r}{\sigma\times \operatorname{id}}\arrow{d}{\pr_{C\times C_c}} & C_{c+1}\times U \arrow{d}{\pr_{C_{c+1}}} \\
C\times C_c \arrow{r}{\sigma}  & C_{c+1} 
\end{tikzcd}
\end{equation*}
We also recall from  Lemma \ref{lemma:globalmayervietoris} that $\sigma^*N_L \cong \pr_{C}^*L \otimes \pr_{C_c}^*N_L \otimes \mathcal{O}(-\Xi_{1,c})$, so that by base change and projection formula we get the following:
\begin{align*}
\mathscr{G}_{|U} &= \pr_{U,*}\left(\pr_{C_{c+1}}^*(E_{K_C-L}\otimes N_L)\otimes \mathcal{O}(-\Xi_{c+1,g-2c})\right) \\
& \cong \pr_{U,*}\left(\pr_{C_{c+1}}^*(\sigma_*(\pr_C^*(K_C-L))\otimes N_L)\otimes \mathcal{O}(-\Xi_{c+1,g-2c})\right) \\
& \cong \pr_{U,*}\left(\pr_{C_{c+1}}^*(\sigma_*(\pr_C^*K_C\otimes \pr_{C_c}^*N_L \otimes \mathcal{O}(-\Xi_{1,c}))\otimes \mathcal{O}(-\Xi_{{c+1},g-2c})\right) \\
& \cong \pr_{U,*}\left( (\sigma\times \operatorname{id})_*(\pr_C^*K_C\otimes \pr_{C_c}^*N_L \otimes \mathcal{O}(-\Xi_{1,c}))\otimes \mathcal{O}(-\Xi_{c+1,g-2c})\right). %\label{eq:globallemma1}
\end{align*}
It is straightforward to check that $(\sigma\times \operatorname{id})^*\mathcal{O}(-\Xi_{3,g-4}) \cong \mathcal{O}(-\Xi_{1,g-4}-\Xi_{2,g-4})$, and then the projection formula again yields 
\begin{align}
\mathscr{G}_{|U} & \cong (\pr_{U,*}\circ (\sigma\times \operatorname{id})_*)\left(\pr_C^*K_C\otimes \pr_{C_c}^*N_L \otimes \mathcal{O}(-\Xi_{1,c}-\Xi_{c,g-2c}-\Xi_{1,g-2c})\right)\\
& \cong \pr_{U,*}\left(\pr_C^*K_C\otimes \pr_{C_c}^*N_L \otimes \mathcal{O}(-\Xi_{1,c}-\Xi_{c,g-2c}-\Xi_{1,g-2c})\right)
\end{align} 
Now we factor the projection $\pr_{U}$ via the projection $\pr_{C_c\times U}$ and we get:
\begin{equation}
%\pr_{C_{g-4},*}&\left(\pr_C^*(B+L)\otimes \pr_{C_2}^*N_L \otimes \mathcal{O}(-\Xi_{2,1})\otimes \mathcal{O}(-\Xi_{g-4,2}) \otimes \mathcal{O}(-\Xi_{g-4,1})\right) \\
%\cong \pr_{C_{g-4},*}&\left[ \pr_{C_{g-4}\times C_2,*}\left(\pr_C^*(B+L)\otimes \pr_{C_2}^*N_L \otimes \mathcal{O}(-\Xi_{2,1})\otimes \mathcal{O}(-\Xi_{g-4,2}) \otimes \mathcal{O}(-\Xi_{g-4,1})\right)\right] \\
\mathscr{G}_{|U} \cong \pr_{U,*}\left[ \pr_{C_{g-2c}\times C_c,*}\left(\pr_C^*(B+L)\otimes  \mathcal{O}(-\Xi_{1,c}-\Xi_{1,g-2c})\right)\otimes \pr_{C_c}^*N_L\otimes \mathcal{O}(-\Xi_{c,g-2c})  \right]
\end{equation}
which is exactly what we want.
\end{proof}

\textbf{Step 2:} Now we would like to globalize the resolution of $M_{E_{K_C-D}\otimes N_{L-D}}$ given in Lemma \ref{lemma:resolution}. To do so, it will be useful to write down some facts about the bundles in \eqref{eq:sheafM}. 

\begin{lemma}\label{lemma:sheafH}
	Let $\mathscr{H}$ be as in \eqref{eq:sheafM}. Then $\mathscr{H}\cong \pr_{U}^*{M_{K_C}}_{|U}$ and $\det \mathscr{H} \cong \pr_U^*N^{\vee}_{K_C}$.
\end{lemma}
\begin{proof}
	By definition, $\mathscr{H} = \pr_{C_c\times U,*}( \pr_C^*K_C\otimes \mathcal{O}(-\Xi_{1,g-2c})) = \pr_{C_c\times U,*}\pr_{C\times U}^*(\pr_C^*K_C\otimes \mathcal{O}(-\Xi_{1,g-2c}))$. The cartesian diagram
	\begin{equation}
	\begin{tikzcd}
	C\times C_c\times U \arrow{r}{\pr_{C\times U}}\arrow{d}{\pr_{C_c\times U}} & C\times U \arrow{d}{\pr_{U}} \\
	C_c\times U \arrow{r}{\pr_U} & U
	\end{tikzcd}
	\end{equation}
	together with flat base change shows that $\mathscr{H}$ is the pullback to $C_c\times U$ of the sheaf $\pr_{U,*}(\pr_C^*K_C\otimes \mathcal{O}(-\Xi_{1,g-2c}))$ on $U$. By the construction of tautological bundles \eqref{evB}, we see that the latter coincides precisely with ${M_{E_{K_C}}}_{|U}$. Moreover, by definition of $U$, we see that the sequence 
	\begin{equation}
	0 \lra {M_{E_{K_C}}}_{|U} \lra H^0(C,K_C)\otimes \mathcal{O}_{U} \lra {E_{K_C}}_{|U} \to 0
	\end{equation}
	is exact on the right. Hence, $\det {M_{E_{K_C}}}_{|U} \cong \det E_{{K_C}}^{\vee} \cong N_{K_C}^{\vee}$, so that $\det \mathscr{H} \cong  \pr_U^*N_{K_C}^{\vee}$.
\end{proof}

\begin{lemma}\label{lemma:detE}
	Let $\mathscr{E}$ be as in \eqref{eq:sheafM}. Then $\det \mathscr{E} \cong \pr_{C_c}^*N_{K_C}\otimes \mathcal{O}(-\Xi_{c,g-2c})$.
\end{lemma}
\begin{proof}
	By definition $\mathscr{E} = \pr_{C_c\times U,*}\left( \pr_C^*K_C\otimes \mathcal{O}(-\Xi_{1,g-2c}) \right)_{|\Xi_{1,c}}$. From section \ref{section:additionmaps} we know that the restriction to $\Xi_{1,c}$ can be constructed as the pullback along the map, \begin{equation}
	(\operatorname{id}, \sigma) \times \operatorname{id}\colon C\times C_{c-1} \times U \to C\times C_c \times U, \qquad (p,q,D) \mapsto (p,p+q,D) 
	\end{equation} 
	and it is straightforward to compute that $((\operatorname{id},\sigma)\times \operatorname{id})^*(\pr_{C}^*K_C\otimes \mathcal{O}(-\Xi_{1,g-2c})) \cong \pr_C^*K_C\otimes \mathcal{O}(-\Xi_{1,g-2c})$. Hence, it follows that $\mathscr{E}$ is the pushforward of $ \pr_C^*K_C\otimes \mathcal{O}(-\Xi_{1,g-2c}) $ along the composition:
	\begin{equation}
	\sigma\times \operatorname{id}\colon C\times C_{c-1} \times U \lra C_c\times U, \qquad (p,q,D) \mapsto (p+q,D).
	\end{equation}
	Now, consider the short exact sequence on $C\times C_c\times U$:
	\begin{equation}
	0 \lra \pr_C^*K_C\otimes \mathcal{O}(-\Xi_{1,g-2c}) \lra \pr_C^*K_C \lra (\pr_C^*K_C)_{|\Xi_{1,g-2c}} \lra 0.
	\end{equation}
	Since the map $\sigma \times \operatorname{id}$ is finite, the pushforward $(\sigma\times \operatorname{id})_*$ is exact, and moreover it is easy to see that $(\sigma \times \operatorname{id})_*(\pr_C^*K_C) \cong \pr_{C_c}^*E_{K_C}$, hence we have an exact sequence of sheaves on $C_c\times U$:
	\begin{equation}
	0 \lra \mathscr{E} \lra \pr_{C_c}^*E_{K_C} \lra \mathcal{F} \lra 0
	\end{equation}
	where  $\mathcal{F} = (\sigma\times \operatorname{id})_*((\pr_C^*K_C)_{|\Xi_{1,g-2c}})$ is a sheaf supported on the irreducible divisor $\Xi_{c,g-2c}$ and it is of rank one over it. Hence we see from Lemma \ref{lemma:determinant} that $\det \mathscr{E} \cong \pr_{C_c}^*N_{K_C} \otimes \mathcal{O}(-\Xi_{c,g-2c})$.
\end{proof}

Now it is easy to generalize Lemma \ref{lemma:resolution}.

\begin{lemma}\label{lemma:globalresolution}
	Assume that 
	\begin{equation}
	\dim W^1_{g-c}(C)=\rho(g,1,g-c) = g-2c-2.
	\end{equation} 
	Then there is an exact complex on $C_c\times U$:
	\begin{small}
		\begin{equation}
		0 \to \begin{matrix} \wedge^{2c}\mathscr{H} \\ \otimes \\  \operatorname{Sym}^{c-1}\mathscr{E}^{\vee}\otimes \pr_{C_c}^*S_{L-K_C} \end{matrix} \to  \dots \to \begin{matrix}\wedge^{c+2}\mathscr{H} \\ \otimes \\ \mathscr{E}^{\vee}\otimes \pr_{C_c}^*S_{L-K_C} \end{matrix} \to \begin{matrix}\wedge^{c+1}\mathscr{H} \\ \otimes \\  \pr_{C_c}^*S_{L-K_C} \end{matrix} \to \mathscr{M} \otimes \pr_{C_c}^*N_L\otimes \mathcal{O}(-\Xi_{c,g-2c}) \to 0 
		\end{equation}
	\end{small}
\end{lemma}
\begin{proof}
	Consider the sequence \eqref{eq:sheafM}: the  map $\mathscr{H} \lra \mathscr{E}$ fails to be surjective precisely along the intersection of the  locus 
	\begin{equation}
	\Sigma' =\{ (\xi,D) \in C_c\times C_{g-2c}  \,|\, h^0(C,D+\xi) \geq 2  \}
	\end{equation}
	with $C_c\times U$. Using the same notation as in the proof of Lemma \ref{lemma:c-1veryampleness}, we see that $\Sigma'$ is clearly isomorphic to the locus $\Sigma = \{ (\xi,H) \in C_2\times C_{g-2} \,|\, h^0(C,H)\geq 2, \, \xi\leq H\}$ under the map
	\begin{equation}
	\Sigma' \lra  \Sigma, \qquad (\xi,D) \mapsto (\xi,D+\xi)
	\end{equation}
	and the proof of Lemma \ref{lemma:resolution} shows also that $\dim \Sigma \leq g-2c-1$. Hence, the degeneracy locus of the map $\mathscr{H} \to \mathscr{E}$ has codimension at least $g-2c+c-(g-2c-1) = c+1$, which is precisely the expected codimension. Hence, the Buchsbaum-Rim complex \cite[Theorem B.2.2]{LazarsfeldPositivityI2004} gives a resolution of $\mathscr{M}$ as follows:  
	\begin{small}
		\begin{equation}
		0 \to \begin{matrix} \wedge^{2c}\mathscr{H} \\ \otimes \\  \operatorname{Sym}^{c-1}\mathscr{E}^{\vee}\otimes \det \mathscr{E}^{\vee} \end{matrix} \to  \dots \to \begin{matrix}\wedge^{c+2}\mathscr{H} \\ \otimes \\ \mathscr{E}^{\vee}\otimes \det \mathscr{E}^{\vee} \end{matrix} \to \begin{matrix}\wedge^{c+1}\mathscr{H} \\ \otimes \\  \det \mathscr{E}^{\vee} \end{matrix} \to \mathscr{M} \to 0 
		\end{equation}
	\end{small}
	If we tensor this resolution by $\pr_{C_c}^*N_{L}\otimes \mathcal{O}(-\Xi_{c,g-2c})$ we observe from Lemma \ref{lemma:detE} that 
	\begin{equation}
	\det\mathscr{E}^{\vee}\otimes \pr_{C_c}^*N_{L}\otimes \mathcal{O}(-\Xi_{c,g-2c}) \cong  \pr_{C_c}^*(N_L - N_{K_C}) \cong \pr_{C_c}^*S_{L-K_C}
	\end{equation}
	so that we obtain precisely the complex in the statement of the lemma.
\end{proof} 

\textbf{ Step 3: } Now we can specialize to the case $c=2$. 

\begin{lemma}\label{lemma:isoG}
	Let $C$ be a smooth curve of genus $g$ and Clifford index $\operatorname{Cliff}(C)\geq 2$ and $L$ a line bundle of degree $2g+p-1$ which is $(p+1)$-very ample. Assume moreover that $C$ is not bielliptic. Then
	\begin{equation}
	\mathscr{G}_{|U} \cong  R^1\pr_{U,*}(\mathscr{E}^{\vee} \otimes \pr_{C_2}^*S_{L-K_C}) \otimes N_{K_C}^{\vee}.
	\end{equation}
\end{lemma}
\begin{proof}
 Since $C$ has Clifford index at least two and is not bielliptic, Mumford's Theorem \eqref{eq:mumford} gives that $\dim W^1_{g-2}(C) = \rho(g,1,g-2)$, so that we can apply Lemma \ref{lemma:globalresolution} and obtain an exact sequence on $C_2\times U$:
\begin{equation}
0 \to \wedge^4 \mathscr{H} \otimes \mathscr{E}^{\vee} \otimes \pr_{C_2}^*S_{L-K_C} \to \wedge^3 \mathscr{H} \otimes \pr_{C_2}^*S_{L-K_C} \to \mathscr{M} \otimes \pr_{C_2}^*N_L \otimes \mathcal{O}(-\Xi_{2,g-4}) \to 0 
\end{equation}
Now we should pushforward this exact sequence along $\pr_{U}$. We first observe that the restriction of $\wedge^3\mathscr{H}\otimes \pr_{C_2}^*S_{L-K_C}$ to a fiber $C_2\times\{D\}$ is given by $\wedge^3 H^0(C,K_C-D)\otimes S_{L-K_C}$, and since $L$ is $(p+1)$-very ample we know from Lemma \ref{lemma:p1veryampleness} and Lemma \ref{lemma:cohomologysymmetricproduct} that $H^0(C_2,S_{L-K_C}) = H^1(C_2,S_{L-K_C}) = 0$. Hence, $\pr_{U,*}( \wedge^3\mathscr{H}\otimes \pr_{C_2}^*S_{L-K_C}) = R^1\pr_{U,*}( \wedge^3\mathscr{H}\otimes \pr_{C_2}^*S_{L-K_C}) \cong 0$ as well, so that pushing forward the previous exact sequence and applying Lemma \ref{lemma:GwithM} we get that
\begin{equation}
\mathscr{G}_{|U} \cong \pr_{U,*}(\mathscr{M}\otimes \pr_{C_2}^*N_L\otimes \mathcal{O}(-\Xi_{2,g-4})) \cong  R^1\pr_{U,*}(\wedge^4 \mathscr{H} \otimes \mathscr{E}^{\vee} \otimes \pr_{C_2}^*S_{L-K_C})
\end{equation}
To conclude, observe that $\mathscr{H}$ is a bundle of rank $4$, so that Lemma \ref{lemma:sheafH} gives that $\wedge^4 \mathscr{H} \cong \det \mathscr{H} \cong \pr_U^*N_{K_C}^{\vee}$, and then we get what we want from the projection formula.
\end{proof}

To conclude, we need to compute explicitly the sheaf $R^1\pr_{U,*}(\mathscr{E}^{\vee} \otimes \pr_{C_2}^*S_{L-K_C})$.

\begin{lemma}\label{lemma:R1pu}
	With the same notations as in Lemma \ref{lemma:isoG} we have
	\begin{equation}
	\mathscr{G}_{|U} \cong \begin{cases} N_{K_C}^{\vee} & \text{ if } g=p+4 \\ 0 & \text{ if } g\ne p+4 \end{cases}
	\end{equation}
\end{lemma}
\begin{proof}
	Thanks to Lemma \ref{lemma:isoG} this amounts to showing that $R^1\pr_{U,*}(\mathscr{E}^{\vee} \otimes \pr_{C_2}^*S_{L-K_C}) \cong \mathcal{O}_U$ when $g=p+4$ and that it vanishes otherwise. Thus, consider the map 
	\begin{equation}
	\sigma\times \operatorname{id} \colon C\times C_{c-1} \times U \lra C_c \times U
	\end{equation} 
	as in the proof of Lemma \ref{lemma:detE}. We keep in mind that for us $c=2$, but we keep the dependency on $c$ explicit for clarity. From the proof of Lemma \ref{lemma:detE}, we know that $\mathscr{E} \cong (\sigma\times \operatorname{id})_*(\pr_C^*K_C\otimes \mathcal{O}(-\Xi_{1,g-2c}))$, and since the map is a finite cover ramified along $\Xi_{1,c-1}$, Grothendieck duality gives that $\mathscr{E}^{\vee} \cong (\sigma\times \operatorname{id})_*(\pr_C^*K_1^{\vee}\otimes \mathcal{O}(\Xi_{1,g-2c}+\Xi_{1,c-1}))$.  We also have $\sigma^*S_{L-K_C}\cong \pr_C^*(L-K_C)\otimes \pr_{C_{c-1}}^*(L-K_C)$, hence the projection formula gives
	\begin{equation}
	\mathscr{E}^{\vee} \otimes \pr_{C_2}^*S_{L-K_C} \cong (\sigma\times \operatorname{id})_*\left( \pr_C^*(L-2K_C)\otimes \pr_{C_{c-1}}^*(L-K_C) \otimes \mathcal{O}(\Xi_{1,g-2c}+\Xi_{1,c-1}))\right)
	\end{equation}
	Since $(\sigma\times \operatorname{id})$ is a finite map, we see from the Leray spectral sequence that 
	\begin{equation}
	R^1\pr_{U,*}(\mathscr{E}^{\vee} \otimes \pr_{C_2}^*S_{L-K_C})\cong R^1\pr_{U,*}( \pr_1^*(L-2K_C)\otimes \pr_{2}^*(L-K_C) \otimes \mathcal{O}(\Xi_{1,g-4}+\Xi_{1,1}))).
	\end{equation}
	To compute the right-hand side,  we factor the projection to $U$ as the composition 
	\begin{equation}
	C\times C \times U \overset{f}{\lra} C \times U \overset{g}{\lra} U ,\qquad f\colon (x,y,D) \mapsto (x,D), \qquad h\colon (x,D)\mapsto D
	\end{equation}
	and we apply the Leray spectral sequence. First we observe that  
	\begin{multline}
	R^i f_*(\pr_{C}^*(L-2K_C)\otimes \pr_{C_{c-1}}^*(L-K_C) \otimes \mathcal{O}(\Xi_{1,g-2c}+\Xi_{1,c-1}))) \\ \cong R^if_*(\pr_{C_{c-1}}^*(L-K_C)\otimes \mathcal{O}(\Xi_{1,c-1}))\otimes \pr_C^*(L-2K_C)\otimes \mathcal{O}(\Xi_{1,g-2c})
	\end{multline} 
	from the projection formula. Moreover the restriction of $\pr_{C_{c-1}}^*(L-K_C)\otimes \mathcal{O}(\Xi_{1,c-1})$ to the fiber $\{x\}\times C \times \{D\}$ of $p$ is given by $L-K_C+x$, and Lemma \ref{lemma:p1veryampleness} and Lemma \ref{kveryample} show that $H^0(C,L-K_C+x) = H^1(C,2K_C-L-x)^{\vee}=0$. In particular, we get that $f_*(\pr_{C_{c-1}}^*(L-K_C)\otimes \mathcal{O}(\Xi_{1,c-1}))=0$  and the Leray spectral sequence shows that 
	\begin{multline}
		R^1\pr_{U,*}( \pr_C^*(L-2K_C)\otimes \pr_{C_{c-1}}^*(L-K_C) \otimes \mathcal{O}(\Xi_{1,g-2c}+\Xi_{1,c-1}))) \\ \cong h_*(R^1f_*(\pr_{C_{c-1}}^*(L-K_C)\otimes \mathcal{O}(\Xi_{1,c-1}))\otimes \pr_{C}^*(L-2K_C)\otimes \mathcal{O}(\Xi_{1,g-2c}) ).
	\end{multline} 
	Now, by Grothendieck duality we see that 
	\begin{equation}
	R^1f_*(\pr_{C_{c-1}}^*(L-K_C)\otimes \mathcal{O}(\Xi_{1,c-1})) \cong (f_*(\pr_{C_{c-1}}^*(2K_C-L)\otimes \mathcal{O}(-\Xi_{1,c-1})))^{\vee}
	\end{equation}
	 and from \eqref{evB} it follows that  $f_*(\pr_{C_{c-1}}^*(2K_C-L)\otimes \mathcal{O}(-\Xi_{1,c-1})) \cong \pr_C^*M_{2K_C-L}$, where $M_{2K_C-L}$ is the kernel bundle sitting in the exact sequence:
	\begin{equation}\label{eq:M2KC-L}
	0 \lra M_{2K_C-L} \lra H^0(C,2K_C-L)\otimes \mathcal{O}_C \lra 2K_C-L \lra 0
	\end{equation}
	 Notice that the sequence is exact on the right because of Lemma \ref{lemma:p1veryampleness}. Hence, we have proved that
	 \begin{equation}\label{eq:isostrano}
	 R^1\pr_{U,*}(\mathscr{E}^{\vee} \otimes \pr_{C_2}^*S_{L-K_C}) \cong h_*(\pr_C^*(M_{2K_C-L}^{\vee}\otimes (L-2K_C) ) \otimes \mathcal{O}(\Xi_{1,g-2c}) ) 
	 \end{equation}
	 Moreover, we see from the sequence \eqref{eq:M2KC-L} and Lemma \ref{lemma:p1veryampleness} that $M_{2K_C-L}$ is a bundle of rank $r = h^0(2K_C-L)-1 = g-p-3$, hence, the previous isomorphism becomes
	 \begin{equation}\label{eq:isostrano2}
	 R^1\pr_{U,*}(\mathscr{E}^{\vee} \otimes \pr_{C_2}^*S_{L-K_C}) \cong h_*(\pr_C^*(\wedge^{g-p-4}M_{2K_C-L}) \otimes \mathcal{O}(\Xi_{1,g-2c}) ).
	 \end{equation}
	 Observe that the restriction of $\pr_C^*(\wedge^{g-p-4}M_{2K_C-L}) \otimes \mathcal{O}(\Xi_{1,g-2c})$ to a fiber $C\times \{D\}$ is given by $\wedge^{g-p-4}M_{2K_C-L}\otimes \mathcal{O}_C(D)$, and standard results on kernel bundles show that $H^0(C,\wedge^{g-p-4}M_{2K_C-L}\otimes \mathcal{O}_C(D))$ is isomorphic to the kernel of the Koszul differential:
	 \begin{equation}
	  \operatorname{Ker} \left[ \wedge^{g-p-4}H^0(C,2K_C-L)\otimes H^0(C,D) \lra \wedge^{g-p-5}H^0(C,2K_C-L)\otimes H^0(C,D+L) \right]
	 \end{equation}
	  In particular, since $h^0(C,D)=1$ by definition of the set $U$, we see that this kernel is trivial whenever $g\ne p+4$, which proves what we want in this case.

	  Suppose instead that $g=p+4$: then \eqref{eq:isostrano2} becomes
	  \begin{equation}
	  R^1\pr_{U,*}(\mathscr{E}^{\vee} \otimes \pr_{C_2}^*S_{L-K_C}) \cong h_*\mathcal{O}(\Xi_{1,g-2c})  
	  \end{equation}
	  We claim that this sheaf is the trivial bundle: indeed, the natural map $\mathcal{O} \to \mathcal{O}(\Xi_{1,{g-2c}})$, induces another morphism $h_*\mathcal{O} \cong \mathcal{O}_U\to h_*\mathcal{O}(\Xi_{1,g-2c})$. By our chouce of $U$ we see that $h_*\mathcal{O}(\Xi_{1,g-2c})$ is a line bundle whose fiber at $D$ is identified with $H^0(C,D)$, and moreover, the map $\mathcal{O}_U \to \mathcal{O}(\Xi_{1,g-2c})$ on the fibers corresponds to the canonical map $H^0(C,\mathcal{O}) \to H^0(C,D)$, which is an isomorphism for each $D$ in $U$. Hence, it follows that $\mathcal{O}_U \to h_*\mathcal{O}(\Xi_{1,g-4})$ is an isomorphism as well.
\end{proof}

\textbf{Step 4:} We can finally give the proof of our Main Theorem

\begin{proof}[Proof of Main Theorem]
	For clarity, we repeat once again the strategy outlined at the end of Section \ref{section:c2}.  Let $C$ be a smooth curve of genus $g$, Clifford index $\operatorname{Cliff}(C)\geq 2$ and not bielliptic. Let also $L$ be a line bundle of degree $2g+p-1$ which is $(p+1)$-very ample. We want to prove that $K_{p,1}(C,L,L)=0$, or by duality, that $K_{g-2,1}(C,K_C-L,L)=0$.  Thanks to the Global Vanishing Criterion of Lemma \ref{cor:basicvanishingsyzygies}, it is enough to show that $H^0(C_{g-4},\mathscr{G}\otimes N_L)=0$, where $\mathscr{G}$ is the pushforward of $\operatorname{pr}_{C_{3}}^*(E_{K_C-L}\otimes N_L)\otimes  \mathcal{O}(-\Xi_{g-4,3})$ from $C_{g-4}\times C_3$. Since $\mathscr{G}$ is torsion-free it is enough to show that $H^0(U,\mathscr{G}\otimes N_L)=0$. However, if $g\ne p+4$, we know from Lemma \ref{lemma:R1pu} that $\mathscr{G}_{|U}=0$, and we conclude. If instead $g=p+4$, we know again from Lemma \ref{lemma:R1pu} that $\mathscr{G}_{|U} \cong N_{K_C}^{\vee}$, so that $H^0(U,\mathscr{G}\otimes N_L) \cong H^0(U,N_{K_C}^{\vee}\otimes N_L) \cong H^0(U,S_{L-K_C})$. By Remark \ref{remark:Ularge}, $U$ is a large open subset, hence $H^0(U,S_{L-K_C}) \cong H^0(C_{g-4},S_{L-K_C}) \cong \operatorname{Sym}^{g-4}H^0(C,L-K_C)$, and Lemma \ref{lemma:p1veryampleness} shows that $H^0(C,L-K_C)=0$. 
\end{proof}

\subsection{Concluding remarks}

It is natural to ask whether the strategy employed here could be used for the next cases of the Secant Conjecture, when the curve $C$ satisfies the Brill-Noether condition
\begin{equation}\label{eq:brillnoethergrowth}
\dim W^1_{g-c}(C) = \rho(g,1,g-c) = g-2c-2.
\end{equation}
For this, the key missing ingredient would be the vanishing statement for the cohomology groups
\begin{equation}
H^i(C_c,\operatorname{Sym}^i E_{K_C-D}^{\vee} \otimes S_{L-K_C}) = 0
\end{equation}
where $D\in C_{g-2c}$ is a general divisor. Similar vanishing statements were recently used in \cite{EinSingularities20} for syzygies of secant of high degree curves and in \cite{AgostiniAsymptotic2020} for surfaces.  Observe that the Brill-Noether condition \eqref{eq:brillnoethergrowth} corresponds to Aprodu's linear growth condition, which ensures the validity of Green's conjecture \cite{AproduRemarks2005} and has recently been explored further by Kemeny \cite{KemenyBetti2018}.

\printbibliography

\end{document}